\documentclass[12pt,twoside,reqno]{amsart}
\usepackage[colorlinks=true,citecolor=blue]{hyperref}
\usepackage{textcase}
\usepackage{amsthm,amsmath,amssymb,yhmath,mathrsfs}
\usepackage{graphicx}
 \usepackage{titlesec}
\titleformat{\section}
{\centering\normalfont\bfseries}{\thesection.}{0.5em}{}
\AtBeginDocument{}
\newtheoremstyle{mystyle}
  {\topsep}
  {\topsep}
  {\normalfont}  {}
  {\normalfont\scshape}  {}  {.5em}  {}
\theoremstyle{mystyle}
\newtheorem{theorem}{Theorem}[section]
\newtheorem{corollary}{Corollary}[section]
\newtheorem{lemma}{Lemma}[section]

\newtheorem{remark}{Remark}[section]
\newtheorem{ques}{Question}[section]

\numberwithin{equation}{section}
\begin{document}
\setcounter{page}{1}

\title[A counterexample to Reich's fixed point problem]
{A counterexample to Reich's fixed point problem}
\author[  L. V. Nguyen ]{ Luong Van  Nguyen}
\maketitle
\vspace*{-0.6cm}

\begin{center}
{\footnotesize {\it
Faculty of Natural Sciences, Hong Duc University \\
565 Quang Trung Street, Hac Thanh Ward, Thanh Hoa Province, Vietnam\\
Email: nguyenvanluong@hdu.edu.vn; luonghdu@gmail.com 

}}\end{center}

\vskip 4mm {\footnotesize \noindent {\bf Abstract.} This note  provides a negative answer to the well-known Reich's problem in fixed point theory.


}

%
%

\section{Introduction and Preliminaries}

Let $(X,d)$ be a metric space. We denote by $P(X)$ the collection of all nonempty subsets of $X$, by $CB(X)$ the collection of all nonempty closed bounded subsets of $X$ and by $K(X)$ the collection of nonempty compact subsets of $X$. The diameter of a  set $A\in P(X)$ is defined by $\mathrm{diam}(A) := \sup\{d(x,y):x,y\in A\}$. The distance from a point $x\in X$ to a  set  $A\in P(X)$ is defined by
$$\mathrm{dist}(x,A) := \inf\{d(x,y): y\in A\}.$$
For nonempty subsets $A,B$ of $X$, the excces of $A$ over $B$ is defined by 
$$e(A,B): = \sup\{d(x,B):x\in A\}.$$
The Hausdorff distance $H$ defined on $CB(X)$ is given by
$$H(A,B) : = \max\{e(A,B), e(B,A)\},\quad A, B\in CB(X).$$
The Hausdorff distance induced by the usual metric on $\mathbb{R}$ is denoted by $H_{\mathbb{R}}$.

The following fixed point result for compact values multi-valued mapping  was proved by Reich \cite{R72} in 1972 .
\begin{theorem}[\cite{R72}] \label{THMR72}
	Let $(X,d)$ be a complete metric space and $T:X\to K(X)$ be a multi-valued mapping satisfying
	\begin{equation}
		\label{RC1}
		H(T(x),T(y)) \le \kappa(d(x,y))d(x,y),\quad x,y\in X, x\ne y,
	\end{equation}
	where $\kappa: (0,\infty)\to [0,1)$  is a function with the property
	\begin{equation}
		\label{RC2}
		\limsup_{s\to t^+}\kappa(s)<1\quad\text{ for all } t>0.
	\end{equation}
	Then, $T$ has a fixed point $x^*\in X$, i.e., $x^*\in T(x^*)$.
\end{theorem}
In 1974, Reich \cite{R74} (see also \cite[Theorem 9]{R83}) posed the following question for nonempty closed bounded values multi-valued mappings satisfying  conditions as in Theorem \ref{THMR72}.
\begin{ques}[\cite{R74}]\label{QR74}
	Let $(X,d)$ be a complete metric space and $T:X\to CB(X)$ be a multi-valued mapping satisfying \eqref{RC1} with  $\kappa: (0,\infty)\to [0,1)$ having the property \eqref{RC2}. Does $T$ have a fixed point in $X$?
\end{ques}
This question  is known as Reich's fixed point problem and it remains unsolved. However, some  partially affirmative answers have been obtained. In 1983, Mizoguchi and Takahashi \cite{MT89} obtained a fixed point theorem for $T:X\to CB(X)$ in which condition \eqref{RC2} is replaced by the stronger condition
$$\limsup_{s\to t^+} \kappa(s)<1\quad\text{for all }t\ge 0.$$
In 1995, Daffer and Kaneko \cite{DK95} gave an alternative and somewhat more straightforward proof for the theorem of Mizoguchi and Takahashi. A simpler proof of that theorem can be found in \cite{X00}. In 1996,  Daffer, Kaneko and Li \cite{DKL96} obtained another positive result when, in addition,  the function $\kappa$ satisfies 
$$\kappa(t)\le 1-at^{\sigma}$$ 
for  all sufficiently small positive $t$, where $a>0$ and $\sigma\in (0,1)$. Also in 1996, Chen \cite{C96} supplied a different sufficient hypothesis by  assuming the following additional condition on the considered mapping $T$:
\begin{itemize}
	\item[] If $Y$ is a closed subset of $X$ such that $T(x)\cap Y\ne \emptyset$ for all $x\in Y$, then $d(x,T(x)) =d(x,T(x)\cap Y)$ for all $x\in Y$.
\end{itemize}
Xu \cite{X00} also emphasized that this condition is restrictive and even constant mappings do not always satisfy it. In 2019, Aze and Corvellec \cite{AC19} proved that if $T$ satisfies  conditions in Question \ref{QR74}, then $\inf\{d(x,T(x)):x\in X\} =0$ and they also gave a partial answer to Question \ref{QR74} with an additional condition that $\kappa$ is nonincreasing and $(1-\kappa)^{-1}$ is summable near zero.

In 2009, Jachymski \cite{J09} posed the  following variant  of  Question \ref{QR74}, with a slightly stronger assumption.
\begin{ques}[\cite{J09}] 
	\label{QJ09}
	Let $(X,d)$ be a complete metric space and $T:X\to CB(X)$ be a multi-valued mapping such that
	\begin{equation}
		\label{JC1}
		H(T(x),T(y))\le \varphi(d(x,y)) \quad\text{for all } x,y\in X,
	\end{equation}
	where $\varphi:[0,\infty)\to [0,\infty)$ is increasing and right continuous and $\varphi(t)<t$ for all $t>0$. Does $T$ has a fixed point?
\end{ques}
Note that the condition for $T$ in Question \ref{QR74} is equivalent to \eqref{JC1} with a right upper semicontinuous function $\varphi:[0,\infty)\to [0,\infty)$ such that $\varphi(t)<t$ for all $t>0$ (see  \cite[Remark 7]{JJ07}). Here, $\varphi:[0,\infty)\to [0,\infty)$  is right upper semicontinuous if
$$\limsup_{s\to t^+}\varphi(s)\le \varphi(t)\quad\text{for all } t\ge 0.$$
Another version of Question \ref{QJ09} was raised before in \cite{J98}, where the function $\varphi$ is supposed to be continuous and increasing and $\varphi(t)<t$ for all $t>0$.  For more Jachymski's work related to Reich’s question, we refer the reader to, e.g., \cite{J95,J98}. 

Recently, D\~ung  \cite{D24} claimed an affirmative answer to Question \ref{QJ09} by proving the following theorem.
\begin{theorem}[\cite{D24}]
	\label{DTHM}
	Let $(X,d)$ be a complete metric space and $T:X\to CB(X)$ be a multi-valued mapping satisfying \eqref{JC1} with $\varphi:[0,\infty)\to [0,\infty)$ being right upper semicontinuous and increasing and $\varphi(t)<t$ for all $t>0$. Then, $T$ has a fixed point.
\end{theorem}

However, the proof of Theorem \ref{DTHM} in \cite{D24}  contains some gaps.  For instance, on page 325, the proof asserts that
\begin{itemize}
	\item[(AS)] the boundedness of $A$, together with $\inf\{\mathrm{dist}(x,T(x)):x\in A\}$, yields the existence of a set $A_0\subset A$ with a prescribed small diameter such that  $\inf\{\mathrm{dist}(x,T(x)):x\in A_0\}$. 
\end{itemize}
Unfortunately, this assertion is not true and hence the proof does not establish Theorem \ref{DTHM} as stated.

In this paper, we present a counterexample to Theorem \ref{DTHM} and therefore provide a negative answer to Reich's fixed poin problem.

\section{A counterexample}

Before providing a negative answer to the Reich's fixed point problem, we present here a technical lemma which will be used in proving the main result.
\begin{lemma}
	\label{Lem}
	Let $\alpha \in (0,1/2]$. For $t\in [-1,1]$, define
	\begin{equation}
		\label{EqLem}
		\Lambda_\alpha(t)=\{\lambda\in [-1,1]: |\lambda- (1-\alpha)t|\ge \alpha\}.
	\end{equation}
	Then, 
	\begin{equation}
		\label{IneqLem}
		H_{\mathbb{R}}(\Lambda_\alpha(s),\Lambda_\alpha(t)) = \min\{(1-\alpha)|s-t|,\alpha\},
	\end{equation}
	for all $t,s\in [-1,1]$.
\end{lemma}
\begin{proof} For each $t\in [-1,1]$, set $b_t = (1-\alpha)t$. Then, by \eqref{EqLem},   
	$$\Lambda_\alpha(t) =\{\lambda\in [-1,1]: |\lambda- b_t|\ge \alpha\}= [-1,b_t-\alpha] \cup [b_t+\alpha,1].$$
	For any $\lambda\in [-1,1]$, we have
	$$
	\mathrm{dist}(\lambda, \Lambda_\alpha(t))= 
	\begin{cases}
		0, & \text{ if } \left|\lambda-b_t\right| \geq \alpha, \\ 
		\min\left\{\lambda-\left(b_t-\alpha\right),\left(b_t+\alpha\right)-\lambda\right\}, &\text{ if }  \left|\lambda-b_t\right|<\alpha .
	\end{cases}
	$$
	Equivalently,
	$$
	\mathrm{dist}(\lambda, \Lambda_\alpha(t))= \max\{0,\alpha - |\lambda-b_t|\}.
	$$
	It follows that $\mathrm{dist}(\lambda, \Lambda_\alpha(t)) \le \alpha$ for any $\lambda\in [-1,1]$ and $t\in [-1,1]$.
	
	Now, let $t,s\in [-1,1]$. We first show that
	\begin{equation}
		\label{E1}
		H_{\mathbb{R}}(\Lambda_\alpha(s),\Lambda_\alpha(t)) \le \min\{(1-\alpha)|t-s|,\alpha\}.
	\end{equation}
	Let $\lambda\in \Lambda_\alpha(s)$. Then, $|\lambda-b_s|\ge \alpha$. Using the triangle inequality, 
	$$|\lambda - b_t|\ge |\lambda - b_s| - |b_s-b_t|\ge \alpha -(1-\alpha)|t-s|.$$
	Thus,
	$$\mathrm{dist}(\lambda,\Lambda_\alpha(t)) = \max\{0,\alpha - |\lambda-b_t|\} \le (1-\alpha)|t-s|.$$
	Hence,
	$$\mathrm{dist}(\lambda,\Lambda_\alpha(t))\le \min\{(1-\alpha)|t-s|,\alpha\},$$
	for every $\lambda\in \Lambda_\alpha(s)$. Thus,
	$$e(\Lambda_\alpha(s),\Lambda_\alpha(t)) = \sup_{\lambda\in \Lambda_\alpha(s)}\mathrm{dist}(\lambda,\Lambda_\alpha(t))\le \min\{(1-\alpha)|t-s|,\alpha\}.$$
	Exchanging $s$ and $t$, we also get $e(\Lambda_\alpha(t),\Lambda_\alpha(s))\le \min\{(1-\alpha)|t-s|,\alpha\}$. Therefore,  \eqref{E1} has been proved.
	
	To show the reverse inequality of \eqref{E1}, we assume without loss of generality that $s\le t$. Then, $b_s\le b_t$.  We have two cases.
	
	\textbf{\textit{Case 1.}}  $b_t-b_s\le \alpha$. In this case, taking $\lambda_0 = b_s +\alpha \in \Lambda_\alpha(s)$, we have
	\begin{align*}
		\mathrm{dist}(\lambda_0,\Lambda_\alpha(t)) &= \min\{\lambda_0-(b_t-\alpha),(b_t+\alpha)-\lambda_0\}\\
		&=\min\{2\alpha -(b_t-b_s), b_t-b_s\}\\
		&= b_t-b_s = (1-\alpha)|t-s|.
	\end{align*}
	Thus, $H(\Lambda_\alpha(s),\Lambda_\alpha(t))\ge (1-\alpha)|t-s|$.
	
	\textbf{ \textit{Case 2.}} $b_t-b_s\ge \alpha$. In this case, taking $\lambda_0 =b_t$, we have $\lambda_0=b_t\ge b_s+\alpha$. Hence, $\lambda_0 \in \Lambda_\alpha(s)$ and
	$$\mathrm{dist}(\lambda_0,\Lambda_\alpha(t))= \alpha.$$
	Thus, $H(\Lambda_\alpha(s),\Lambda_\alpha(t))\ge \alpha$.
	
	Combining both cases, we have
	\begin{equation}
		\label{E2}
		H(\Lambda_\alpha(s),\Lambda_\alpha(t))\ge \min\{(1-\alpha)|t-s|,\alpha\}.
	\end{equation}
	Therefore, \eqref{EqLem} follows from \eqref{E1} and \eqref{E2}. The proof is complete.
\end{proof} 

The main result of this note is state as follows.
\begin{theorem}
	\label{MRS}
	Let $X = \{x = \{x_n\}\in c_0:  ||x||_{\infty}\le 1\}$  equiped with the metric $d(x,y) = ||x-y||_{\infty}$ for all $x,y \in X$ and let $T:X\rightrightarrows X$ be defined by
	\begin{equation}
		\label{DefT}
		T(x) = \left\{z=\{z_n\}\in X:  \,\left|(n+1)z_n - nx_n  \right| \ge 1\,\,\text{ for all }\,n\ge 1    \right\}
	\end{equation}
	for all $x = \{x_n\}\in X$. Then, the following assertions hold.
	\begin{itemize}
		\item[\rm (i)]  $(X,d)$ is a complete metric space.
		\item[\rm (ii)] $T(x)\in CB(X)$ for all $x\in X$.
		\item[\rm(iii)] For  every $x\in X$, the set $T(x)$ is not compact.
		\item[\rm(vi)] For all $x,y\in X$,
		\begin{equation}
			\label{IneqH}
			H(T(x),T(y)) \le \frac{d(x,y)}{1+d(x,y)}.
		\end{equation}
		\item[\rm (v)]  $\inf\{\mathrm{dist}(x,T(x)):x\in X\} = 0$ and   $\inf\{\mathrm{dist}(x,T(x)):x\in A_0\} > 0$ for any $r\in (0,1)$ and for any nonempty subset $A_0$ of $X$ with $\mathrm{diam}A_0\le r$.
		\item[\rm (vi)] $T$ is fixed-point free in $X$. 
	\end{itemize}
\end{theorem}
\begin{proof} (i) Since $(c_0,d)$ is a complete metric space and $X$ is a closed subset of $c_0$, the metric space $(X,d)$ is complete.
	
	(ii)  Let $x = \{x_n\} \in X$. We first show that $T(x)\ne \emptyset$. Define $z=(z_n)$ by 
	\begin{equation}
		\label{ETH0}
		z_n= 
		\begin{cases}
			0, & \text{ if } \left|x_n\right| \geq \dfrac{1}{n}, \\[8pt]
			\dfrac{1}{n+1}(nx_n-1), & \text{ if } 0 \leq x_n<\dfrac{1}{n}, \\[8pt]
			\dfrac{1}{n+1}(nx_n+1),  & \text{ if }  -\dfrac{1}{n}<x_n<0 .
		\end{cases}
	\end{equation}
	If $|x_n| <1/n$, then 
	$$|z_n| \le  \frac{n|x_n|+1}{n+1}<\frac{2}{n+1},$$
	Hence, 
	$$|z_n|<\frac{2}{n+1}\quad\text{ for all } n\ge 1.$$
	It follows that $z_n\to0$ as $n\to\infty$ and
	$$||z||_{\infty} \le \sup_{n\ge 1} \frac{2}{n+1} = 1.$$
	Thus, $z\in X$. 
	
	We next show that $z\in T(x)$. Indeed, if $|x_n|\ge 1/n$, then
	$$\left|(n+1)z_n-nx_n\right| = n|x_n| \ge n\cdot \frac{1}{n} = 1.$$
	If $0\le x_n <1/n$, then
	$$\left|(n+1)z_n-nx_n\right| = \left|(nx_n-1)-nx_n\right| = 1.$$
	If $-1/n<x_n<0$, then 
	$$\left|(n+1)z_n-nx_n\right| = \left|(nx_n+1)-nx_n\right| = 1.$$
	So, in all cases, we have
	$$\left|(n+1)z_n-nx_n\right|\ge  1$$
	for all $n\ge 1$. Thus, $z\in T(x)$ and therefore $T(x)$ is nonempty.
	
	Assume now that $z^{(k)} = \{z_n^{(k)}\} \in T(x)$ and $z^{(k)}$ converges to some $z = \{z_n\}$ in $X$.  Since $z^{(k)}\in T(x)$, we have for all $n\ge 1$ that
	$$\left|(n+1)z_n^{(k)} - nx_n\right| \ge 1.$$
	Since $z^{k}\to z$, we have $z_n^{(k)}\to z_n$ as $k\to\infty$ for all $n\ge 1$. Taking the limit as $k\to\infty$ on both sides of the latter ine quality, we get
	$$\left|(n+1)z_n - nx_n\right| \ge 1.$$
	Thus, $z\in T(x)$ and therefore $T(x)$ is closed.
	
	Moreover,  for every $z,z'\in T(x)$, we have
	$$d(z,z')=||z-z'||_{\infty}\le ||z||_{\infty} +||z'||_{\infty}\le 2.$$
	Thus, $T(x)$ is bounded. So, we have verified that $T(x)\in CB(X)$ for all $x\in X$.
	
	(iii) Let $x=\{x_n\}\in X$ and define $z=\{z_n\}$ as in \eqref{ETH0}. Then, $z\in T(x)$. Moreover,
	\begin{equation}
		\label{ETH01} 
		|z_n|\le \frac{1}{n+1}
	\end{equation}
	for all $n\ge 1$. Indeed, if  $|x_n|\ge 1/n$, then $z_n=0$ and \eqref{ETH01} holds. If $0\le x_n<1/n$, then
	$ -1/(n+1)\le z_n<0$   and \eqref{ETH01} holds.  If  $-1/n<x_n<0$, then
	$0<z_n<1/(n+1)$ and \eqref{ETH01} holds. Therefore, \eqref{ETH01} holds for all $n\ge 1$.
	
	Now, for each $k\ge 1$,  let $z^{(k)}$ be the sequence obtained by replacing  the $k$th coordinate of $z$ by $1$, leaving all other coordinates unchanged. Then, $z^{(k)}\in X$ for all $k$. Moreover, for each $k\ge 1$,
	$$|(k+1)z_k^{(k)} - kx_k| = |k+1-kx_k|\ge k+1 - k|x_k|\ge 1.$$
	Hence, $z^{(k)}\in T(x)$ for all $k$. For distinct $k$ and $n$,
	$$d(z^{(k)},z^{(n)}) \ge |1-z_n| \ge 1-\frac{1}{n+1}\ge \frac{1}{2}.$$
	Thus, $\{z^{(k)}\}$ has no Cauchy subsequence. So, $T(x)$ is not compact.

	(iv) For each $n$, set $\alpha_n = 1/(n+1)$. Then, $\alpha_n\in (0,1/2]$ for all $n\ge 1$ and
	\begin{align*}
		T(x)&=\left\{ z = \{z_n\}\in X: \left|z_n - \frac{n}{n+1}x_n\right|\ge \frac{1}{n+1}\,\,\text{ for all }\, n\ge 1\right\}\\
		&= \left\{z=\{z_n\}\in X: |z_n-(1-\alpha_n)x_n|\ge \alpha_n \,\,\text{ for all }\, n\ge 1\right\}\\
		&= \left\{z=\{z_n\}\in X: z_n \in \Lambda_{\alpha_n}(x_n) \,\,\text{ for all }\, n\ge 1\right\}.
	\end{align*}
	Let $x= \{x_n\}, y =\{y_n\}\in X$ and take an arbitrary element $u = \{u_n\}\in T(x)$. Since each set $\Lambda_{\alpha_n}(y_n)$ is a nonempty compact subset of $\mathbb{R}$, there exists $v_n\in \Lambda_{\alpha_n}(y_n)$ such that
	$$|u_n-v_n| = \mathrm{dist}(u_n,\Lambda_{\alpha_n}(y_n)).$$
	Since $u_n \in \Lambda_{\alpha_n}(x_n)$, by Lemma \ref{Lem}, we have
	\begin{align}
		\label{ETH1}
		|u_n-v_n|&=\mathrm{dist}(u_n,\Lambda_{\alpha_n}(y_n))\nonumber\\
		&\le H_{\mathbb{R}}(\Lambda_{\alpha_n}(x_n),\Lambda_{\alpha_n}(y_n))\nonumber\\
		&=\min\{(1-\alpha_n)|x_n-y_n|,\alpha_n\}\nonumber\\
		&\le \min\{(1-\alpha_n)d(x,y),\alpha_n\}.
	\end{align}
	It follows that
	$$|v_n|\le |u_n| + |v_n-u_n| \le |u_n| + \alpha_n.$$
	Since $u=\{u_n\}\in T(x)$, so $u_n\to 0$, and $a_n\to 0$, we have $v_n\to 0$ as $n\to\infty$.  Moreover, since $v_n\in \Lambda_{\alpha_n}(y_n)\subset [-1,1]$, one has $||v||_{\infty}\le 1$. Consequently, $v=\{v_n\}\in X$ and therefore $v\in T(y)$.
	
	Furthermore, for each $n$, we have
	\begin{equation}
		\label{ETH2}
		\min\{(1-\alpha_n)d(x,y),\alpha_n\}\le \frac{d(x,y)}{1+d(x,y)}.
	\end{equation}
	Indeed, if $\alpha_n \le d(x,y)/[1+d(x,y)]$, then \eqref{ETH2} holds. Otherwise, $\alpha_n >d(x,y)/[1+d(x,y)]$, then 
	$$(1-\alpha_n)d(x,y)< \left[1- \frac{d(x,y)}{1+d(x,y)}\right]d(x,y) = \frac{d(x,y)}{1+d(x,y)}.$$
	and hence \eqref{ETH2} holds. This proves  \eqref{ETH2}.
	
	Now,
	$$||u-v||_{\infty}=\sup_{n}|u_n-v_n| \le \sup_{n} \min\{(1-\alpha_n)d(x,y),\alpha_n\} \le \frac{d(x,y)}{1+d(x,y)}.$$
	Since $v\in T(y)$, we have
	$$\mathrm{dist}(u,T(y))\le \inf_{z\in T(y)}||u-z||_{\infty}\le ||u-v||_{\infty}\le \frac{d(x,y)}{1+d(x,y)}.$$
	As $u\in T(x)$ is arbitrary, 
	$$e(T(x),T(y)) = \sup_{u\in T(x)}\mathrm{dist}(u,T(y))\le \frac{d(x,y)}{1+d(x,y)}.$$
	Repeating the argument with $x,y$ interchanged, we also get
	$$e(T(y),T(x))\le \frac{d(x,y)}{1+d(x,y)}.$$
	Thus,
	$$H(T(x),T(y))\le \frac{d(x,y)}{1+d(x,y)}.$$
	
	(v) For any $x=\{x_n\}\in X$ and $z=\{z_n\}\in T(x)$, we have for all $n\ge 1$ that 
	\begin{align*}
		1\le |(n+1)z_n-nx_n|&\le (n+1)|x_n-z_n| + |x_n|\\
		&\le (n+1)||x-z||_{\infty}+|x_n|.
	\end{align*}
	It follows that
	$$||x-z||_{\infty}\ge \frac{1-|x_n|}{n+1}\quad\text{for all } n\ge 1.$$
	Taking the infimum over $z\in T(x)$, we get
	\begin{equation}
		\label{ETH3}
		\mathrm{dist}(x,T(x))\ge \frac{1-|x_n|}{n+1}
	\end{equation}
	Now let $x= \{x_n\}\in X$ and define  $z=\{z_n\} $  by
	$$
	z_n= 
	\begin{cases}
		\dfrac{nx_n-1}{n+1}, & \text{ if } x_n\ge 0, \\[8pt]
		\dfrac{nx_n+1}{n+1},  & \text{ if }  x_n<0 .
	\end{cases}
	$$
	Then, $z_n\in \Lambda_{\alpha_n}(x_n)$  and
	$$|z_n| \le \frac{n}{n+1}|x_n|+\frac{1}{n+1},$$
	for all $n\ge 1$. Since $x\in c_0$, $x_n\to 0$ as $n\to\infty$. It follows from the latter inequality that $z_n\to 0$ as $n\to \infty$.  Moreover, $|x_n|\le 1$ for all $n\ge 1$, we have $|z_n|\le 1$ for all $n\ge 1$ and thus $||z||_{\infty}\le 1$.  Therefore, $z\in T(x)$. Moreover, if  $x_n\ge 0$, then  
	$$|x_n-z_n|=\left|x_n - \frac{nx_n+1}{n+1}\right| = \frac{1-x_n}{n+1} = \frac{1-|x_n|}{n+1},$$
	and  if  $x_n<0$, then   
	$$|x_n-z_n|=\left|x_n - \frac{nx_n-1}{n+1}\right| = \frac{1+x_n}{n+1} = \frac{1-|x_n|}{n+1}.$$
	Thus,
	$$|x_n-z_n| = \frac{1-|x_n|}{n+1}\quad\text{ for all } n\ge 1,$$
	and hence
	$$||x-z||_{\infty} = \sup_{n\ge 1} \frac{1-|x_n|}{n+1}.$$
	Consequently,
	\begin{equation}
		\label{ETH4}
		\mathrm{dist}(x,T(x)) \le \sup_{n\ge 1} \frac{1-|x_n|}{n+1}.
	\end{equation}
	From \eqref{ETH3} and \eqref{ETH4}, we have
	\begin{equation}
		\label{ETH5}
		\mathrm{dist}(x,T(x)) = \sup_{n\ge 1} \frac{1-|x_n|}{n+1} \quad\text{ for all } x\in X.
	\end{equation}
	Consider the sequence $\{x^{(k)}\}$ in $X$ defined by
	$$
	x_n^{(k)}= 
	\begin{cases}
		1, & \text{ if } n\le k, \\ 
		0,  & \text{ if }  n>k .
	\end{cases}
	$$
	Then,  using  \eqref{ETH5}, we have
	$$ \mathrm{dist}(x^{(k)},T(x^{(k)})) =  \sup_{n>k}\frac{1}{n+1} = \frac{1}{k+2} \to 0\, \text{ as } \,k\to\infty.$$
	Thus,
	$$\inf\{\mathrm{dist}(x,T(x)):x\in X\} = 0.$$
	We now fix $r\in (0,1)$ and any nonempty subset $A_0$ of $X$ with $\mathrm{diam}A_0\le r$.  Take $a=\{a_n\}\in A_0$. Since $a_n\in c_0$, $a_n\to 0$ as $n\to\infty$. Then, there exists $n_0\ge 1$ such that
	$$|a_{n_0}|<\frac{1-r}{2}.$$
	For any $x=\{x_n\}\in A_0$, we have
	\begin{align*}
		|x_{n_0}|&\le |x_{n_0}-a_{n_0}|+ |a_{n_0}|\le ||x-a||_{\infty}+|a_{n_0}|\\
		&<r+ \frac{1-r}{2} = \frac{1+r}{2}.
	\end{align*}
	By \eqref{ETH5}, we have
	$$d(x,T(x))= \sup_{n\ge 1}\frac{1-|x_n|}{n+1}\ge \frac{1-|x_{n_0}|}{n_0+1} >\frac{1-r}{n_0+1}.$$
	Thus,
	$$\inf\{\mathrm{dist}(x,T(x)):x\in A_0\} \ge  \frac{1-r}{n_0+1}>0.$$
	(vi)  Assume that $T$ has a fixed point $x= \{x_n\}\in X$. Then, $x\in T(x)$ and hence  $|(n+1)x_n - nx_n|\ge 1$ or, equivalently, $|x_n|\ge 1$ for all $n\ge 1$. This sequence $\{x_n\}$ cannot converge to zero which contradicts $x\in c_0$. Therefore, $T$ is fixed-point free in $X$.
\end{proof} 
\begin{remark} The mappping $T$ in Theorem \ref{MRS} satisfies all conditions in Theorem \ref{DTHM} with 
	$$\varphi(t) = \frac{t}{1+t}\quad\text{ for all } t\ge 0.$$
	However, the assertion (AS) is not true for $A=X$ and $T$ has no fixed point.
\end{remark}
The following corollary provides a negative answer to Question \ref{QR74}.
\begin{corollary} There exist a complete metric space $(X,d)$, a function $\kappa: (0,\infty) \to [0,1)$ with 
	$$\limsup_{s\to t^+}\kappa(s)<1\quad\text{ for all } t\in (0,\infty),$$ 
	and a mapping $T:X\to CB(X)$ satisfying
	$$H(T(x),T(y)) \le \kappa(d(x,y))d(x,y) \quad\forall x,y\in X, x\ne y,$$
	such that $\mathrm{Fix}(T) =\emptyset$. 
\end{corollary}
 \subsection*{Acknowledgements} During the preparation of this work, the author used OpenAI ChatGPT for detecting gaps in the proof of the main result in \cite{D24}, literature search, and assistance in conceptualizing the example. All literature identified with AI assistance was independently checked against the original sources, and all AI-assisted mathematical arguments  were independently reviewed, verified, and revised by the author. The author takes full responsibility for the accuracy, originality, and integrity of the final manuscript. The authors would like to thank the members of the Mathematics Research Group at Hong Duc University for their fruitful discussions and valuable comments during the preparation of this work.
 \subsection*{Data availability} Data sharing not applicable to this article as no datasets were generated or analysed during the current study.
 \subsection*{Declarations}
 \subsection*{Conflict of interest} The authors declare that they have no conflict of interest.
 

\end{document}